\documentclass[reqno, 12pt]{amsart}
\usepackage{amsmath,amssymb,tikz,stix2}
\usepackage[colorlinks=true,
linkcolor=red,
citecolor=green,
urlcolor=green]{hyperref}

\newtheorem{thm}{Theorem}
\newtheorem{lem}[thm]{Lemma}
\newtheorem{cor}[thm]{Corollary}
\newtheorem{prop}[thm]{Proposition}

\theoremstyle{definition}
\newtheorem{example}[thm]{Example}

\theoremstyle{remark}
\newtheorem{remark}[thm]{Remark}

\newcommand{\D}{\mathbb D}
\newcommand{\T}{\mathbb T}

\newcommand{\E}{\mathcal E}
\newcommand{\HH}{\mathcal H}

\DeclareMathOperator{\Hol}{Hol}
\DeclareMathOperator{\ord}{ord}

\begin{document}
	
	\title[Beurling-type theorem for powers of Bergman shifts]
	{The Beurling-type theorem for powers of weighted Bergman shifts}
	
	\author{Shibo Xu}
	\address[Shibo Xu]{
		School of Mathematical Sciences, Dalian University of Technology,
		Dalian, Liaoning 116024, P. R. China}
	\email{shiboxu98@163.com}
	
		\author{Tao Yu}
	\address[Tao Yu]{
		School of Mathematical Sciences, Dalian University of Technology,
		Dalian, Liaoning 116024, P. R. China}
	\email{tyu@dlut.edu.cn}
	
	\thanks{Yu is supported by the National Natural Science Foundation
		of China (Grant No.~ 12571132).
	}

	
	
	\begin{abstract}
		Let $\mathcal{B}_{\alpha}$ denote the shift operator on the weighted Bergman space
		$A^2_\alpha(\mathbb{D})$.
	    We prove that, for every integer $k\geq1$, the Beurling-type theorem holds for $\mathcal{B}_{\alpha}^k$ on $A^2_\alpha(\mathbb{D})$
		if and only if $-1<\alpha\leq1$. We also show that, for every $\alpha>-1$,
		there exists a finite Blaschke product $B$ such that $M_{B}$ admits a cyclic
		invariant subspace but fails the wandering subspace property.
	\end{abstract}
	
	\maketitle

	\section{Introduction}
	Let $T$ be a bounded linear operator on a Hilbert space $\mathcal H$, a subspace of a Hilbert space is called a wandering subspace for $T$ if it is orthogonal to its images under positive powers of the operator $T$.
	If	$\mathcal M$ be a closed $T$-invariant subspace, then clearly 
	$
	\mathcal M\ominus T\mathcal M
	$
	is wandering for $T$, and we refer to this subspace as the wandering subspace of $\mathcal M$.
	For $\mathcal E\subset\mathcal H$, let $ [\mathcal E]_T$ denote the $T$-invariant subspace generated by $\mathcal E$, i.e.
	\[
	[\mathcal E]_T
	=
	\overline{\operatorname{span}}
	\{T^n f:n\ge0,\ f\in\mathcal E\}.
	\]
	We say that $\mathcal M$ has the wandering subspace property if
	\begin{equation}\label{wsp}
		\mathcal M=[\mathcal M\ominus T\mathcal M]_T.
	\end{equation}
	The Beurling-type theorem is said to hold for $T$ on $\mathcal H$ if
	\eqref{wsp} holds for every closed $T$-invariant subspace $\mathcal M$.
	
	A canonical example is the Hardy shift
	$M_z$ on $H^2(\mathbb D)$. Beurling’s theorem \cite{B1948} identifies every nonzero closed invariant subspace as $\theta H^2(\mathbb D)$ for some inner function $\theta$, with the one dimensional wandering subspace $\mathbb C\theta$.
	 For pure isometries $T$ acting on a Hilbert space $\mathcal H$, the classical Wold--Kolmogorov decomposition states that the operator is	unitarily equivalent to a unilateral shift and the Beurling-type theorem is hold for $T$. 
	Beyond the isometric setting, Shimorin \cite{Shimorin2001} extended the Wold--Kolmogorov decomposition
	to certain classes of operators close to isometries, including concave operators.
	He showed that such operators admit a decomposition into
	a unitary part and an invariant subspace generated by
	$\mathcal H\ominus T\mathcal H$.
	In particular, analytic operators in these classes
	possess the wandering subspace property.

	Let $dA(z)=\pi^{-1}\,dx\,dy$ denote the normalized Lebesgue area measure on the open unit disk  $\mathbb{D}=\{z\in \mathbb{C}: |z|<1\}$. For $\alpha>-1$, the Bergman space $A^2_\alpha(\D)$ consists of all functions on $\D$ satisfying 
	\[
   A^2_\alpha(\D)
	=\bigl\{
	f\in\Hol(\D):\|f\|_\alpha<\infty
	\bigr\},
	\]
	where
	\[
	\|f\|_\alpha^2
	=(\alpha+1)\int_\D
	|f(z)|^2(1-|z|^2)^\alpha\,dA(z),
	\]
	and $\Hol(\D)$ denote the space of holomorphic functions on $\D$.

	Let $\mathcal{B}_{\alpha}$ denote the shift operator on $A^2_\alpha(\D)$.
	Unlike the Hardy shift, $\mathcal{B}_{\alpha}$ is not an isometry, hence the classical Wold--Kolmogorov decomposition does not apply directly. Nevertheless, Aleman, Richter, and Sundberg \cite{ARS1996}
	proved that the Beurling-type theorem holds for the shift
	$\mathcal{B}_{0}$ on the unweighted Bergman space $A^2_0(\D)$. Shimorin \cite{Shimorin2003} extended this result to $A^2_\alpha(\D)$ for $-1<\alpha\leq1$. More recently, Lin, Xu, and Yu \cite{LinXuYu2026}
	proved that the range $-1<\alpha\leq1$ is sharp.
	
Motivated by the work of Lin, Xu, and Yu
\cite{LinXuYu2026}, we investigate whether
	the range $-1<\alpha\leq1$ remains sharp
for the powers $\mathcal{B}_{\alpha}^k$, $k\geq2$.
This question does not follow directly from
the case $k=1$, since an $\mathcal{B}_{\alpha}^k$-invariant
subspace need not be $\mathcal{B}_{\alpha}$-invariant.
		Seco \cite{S2019} constructed $z^k$-invariant
		subspaces that fails the wandering subspace property
		in the Dirichlet-type spaces.
Related developments concerning Bergman-type
reproducing kernels and wandering subspace
theorems can be found in
\cite{McCulloughRichter2002,Olofsson2005}.
	
We now state our main result, which gives
a complete answer to this question.
	
	\begin{thm}
		\label{t1}
		Let $\mathcal{B}_{\alpha}$ be the shift operator on the weighted
		Bergman space $A^2_\alpha(\D)$, where $\alpha>-1$.
		Then, for every integer $k\geq1$,
	   the Beurling-type theorem holds for $\mathcal{B}_{\alpha}^k$
		if and only if $-1<\alpha\leq1$.
	\end{thm}
	
The wandering subspace property for more general
multipliers was investigated by Carswell, Duren,
and Stessin \cite{CDS2002}, who constructed infinite
Blaschke products failing this property on the
classical Bergman space.
 Carswell and
Weir \cite{CW2013} studied wandering subspaces
associated with Blaschke products on the Bergman space.
Gallardo-Guti\'errez, Partington, and Seco
\cite{GPS2020} subsequently investigated finite
Blaschke multipliers on Dirichlet-type spaces.
They established positive results under suitable
equivalent Hilbert norms and, for powers of the shift,
under the original norm in a restricted parameter range.
They also noted that the wandering subspace property
for finite Blaschke products on the classical
Bergman space remained an open problem.

Since the wandering subspace depends on the
 inner product, equivalent renormings
need not preserve the wandering subspace property.
For finite Blaschke products, we obtain the
following result for the  weighted
Bergman norms.

\begin{thm}\label{blaschke}
	For every $\alpha>-1$, there exists a finite
	Blaschke product $B$ such that the Beurling-type
	theorem fails for $M_B$ on $A^2_\alpha(\D)$.
\end{thm}

The remainder of the paper is organized as follows.
Section~\ref{s2} establishes the wandering subspace
property for $\mathcal{B}_{\alpha}^k$ when
$-1<\alpha\leq1$ via a cyclic lifting argument.
Section~\ref{s3} proves its failure for $\alpha>1$
using reproducing kernels and compact perturbations.
Section~\ref{s4} proves Theorem~\ref{blaschke}
by constructing finite Blaschke products with
cyclic invariant subspaces  but fails the wandering subspace property.

	\section{The positive range: cyclic lifting of Shimorin's criterion}\label{s2}
	
	Recall that an operator $T$ on a Hilbert space $\HH$
	is called an analytic contraction if
	\[
	\|T\|\leq 1
	\quad\text{and}\quad
	\bigcap_{n\geq0}T^n\HH=\{0\}.
	\]
	We need the following criterion of Shimorin
	\cite[Theorem~2.5]{Shimorin2003}.
	
	\begin{thm}[Shimorin's criterion]\label{t2}
		Let $T$ be an analytic contraction on a Hilbert space
		$\mathcal H$. Suppose that there exists
		$\varphi\in H^\infty(\D)$ satisfying the following conditions:
		\begin{enumerate}
			\item[\textup{(i)}]
			$\|\varphi(\zeta T)f\|\geq\|f\|$
			for every $\zeta\in\T$ and $f\in\mathcal H$;
			
			\item[\textup{(ii)}]
			$\varphi(0)=0$, and for each $r\in[0,1)$,
			the function
			\[
			\varphi_r(z)=\frac{\varphi(z)-\varphi(r)}{z-r}
			\]
			is bounded away from zero on $\D$;
			
			\item[\textup{(iii)}]
			$|\varphi(r)|<1$ for every $r\in[0,1)$.
		\end{enumerate}
		Then
		\[
		\mathcal H=[\mathcal H\ominus T\mathcal H]_T.
		\]
	\end{thm}

	The following proposition extends the resulting
	wandering subspace property to arbitrary finite powers of $T$.

	\begin{prop}[Finite cyclic lifting]\label{p1}
		Let $T$ be an analytic contraction on a Hilbert space
		$\mathcal H$.
		If $T$ and $\varphi\in H^\infty(\D)$ satisfy
		the hypotheses of Theorem~\ref{t2}, then for every
		integer $k\geq1$, every closed $T^k$-invariant
		subspace $\mathcal M$ satisfies
		\[
		\mathcal M=
		[\mathcal M\ominus T^k\mathcal M]_{T^k}.
		\]
	\end{prop}

	\begin{proof}
		The case $k=1$ follows directly from
		Theorem~\ref{t2} applied to $T|_{\mathcal M}$.
		We now consider the case $k\geq2$.
		
		\medskip
		\noindent
		\emph{Step 1. Construction of the cyclic operator.}
		
		On the orthogonal direct sum $\mathcal H^{\oplus k}$,
		define
		\begin{equation}\label{e5}
			R\mathbf{x}
			=
			(Tx_{k-1},Tx_0,\ldots,Tx_{k-2}),
		\end{equation}
		and the cyclic permutation operator $P$ on
		$\mathcal H^{\oplus k}$ by
		\[
		P\mathbf{x}
		=
		(x_{k-1},x_0,\ldots,x_{k-2}),
		\]
		where
		$\mathbf{x}=(x_0,\ldots,x_{k-1})
		\in\mathcal H^{\oplus k}$.
		Then $P$ is unitary, $P^k=I$, and
		\[
		R=PT^{\oplus k}=T^{\oplus k}P,
		\]
		where $T^{\oplus k}=T\oplus\cdots\oplus T$.
		Consequently,
		\begin{equation}\label{e6}
			R^n=P^n(T^n)^{\oplus k},
			\qquad
			R^k=(T^k)^{\oplus k}.
		\end{equation}
		
		Let $\omega=e^{2\pi i/k}$ and define the unitary
		operator $U$ on $\mathcal H^{\oplus k}$ by
		\[
		(U\mathbf{x})_j
		=
		\frac{1}{\sqrt{k}}
		\sum_{\ell=0}^{k-1}
		\omega^{j\ell}x_\ell,
		\qquad 0\leq j\leq k-1.
		\]
		A direct computation gives
		\[
		UPU^*
		=
		\bigoplus_{j=0}^{k-1}
		\omega^j I_{\mathcal H}.
		\]
		Since $U$ commutes with $T^{\oplus k}$,
		we obtain
		\[
		URU^*
		=
		\bigoplus_{j=0}^{k-1}
		\omega^jT.
		\]
		In particular, $R$ is a contraction.
		
		Since $T$ is analytic, \eqref{e6} gives
		\[
		\bigcap_{n\geq0}
		R^n\mathcal H^{\oplus k}
		=
		\left(
		\bigcap_{n\geq0}T^n\mathcal H
		\right)^{\oplus k}
		=
		\{0\}.
		\]
		Thus $R$ is also analytic and hence completely
		nonunitary, so its $H^\infty$ functional calculus
		is well defined.
		
		The $H^\infty$ functional calculus preserves
		unitary equivalence and finite orthogonal
		direct sums. Therefore, for every $\eta\in\T$,
		\[
		U\varphi(\eta R)U^*
		=
		\bigoplus_{j=0}^{k-1}
		\varphi(\eta\omega^jT).
		\]
		For
		$\mathbf{x}\in\mathcal H^{\oplus k}$,
		write
		\[
		\mathbf{y}
		=
		U\mathbf{x}
		=
		(y_0,\ldots,y_{k-1}).
		\]
		Since $\eta\omega^j\in\T$, the hypothesis on
		$T$ yields
		\[
		\begin{aligned}
			\|\varphi(\eta R)\mathbf{x}\|^2
			&=
			\|U\varphi(\eta R)U^*\mathbf{y}\|^2\\
			&=
			\sum_{j=0}^{k-1}
			\|\varphi(\eta\omega^jT)y_j\|^2\\
			&\geq
			\sum_{j=0}^{k-1}\|y_j\|^2\\
			&=
			\|\mathbf{x}\|^2.
		\end{aligned}
		\]
		Hence
		\begin{equation}\label{e7}
			\|\varphi(\eta R)\mathbf{x}\|
			\geq\|\mathbf{x}\|,
			\qquad
			\eta\in\T,\quad
			\mathbf{x}\in\mathcal H^{\oplus k}.
		\end{equation}
		
		\medskip
		\noindent
		\emph{Step 2. Application of Shimorin's criterion.}
		
		Let $\mathcal M$ be a closed $T^k$-invariant
		subspace and set
		\begin{equation}\label{e8}
			\mathcal N
			=
			\mathcal M\oplus T\mathcal M
			\oplus\cdots\oplus T^{k-1}\mathcal M
			\subset\mathcal H^{\oplus k}.
		\end{equation}
		Since $T$ is bounded below, and hence $\mathcal N$ is a closed
		subspace of $\mathcal H^{\oplus k}$.
		By the definition of $R$,
		\begin{equation}\label{e9}
			R\mathcal N
			=
			T^k\mathcal M\oplus T\mathcal M
			\oplus\cdots\oplus T^{k-1}\mathcal M
			\subset\mathcal N.
		\end{equation}
		Therefore $\mathcal N$ is invariant under $R$.
		
		Since $R$ is  contraction and analytic, so is $R|_{\mathcal N}$.

		We next verify that $R|_{\mathcal N}$ satisfies  the expansive inequality
		\eqref{e7}.
		Since $\mathcal N$ is invariant under $R$,
		the $H^\infty$ functional calculus gives
		\[
		\varphi(\eta R|_{\mathcal N})
		=
		\varphi(\eta R)|_{\mathcal N},
		\qquad \eta\in\T.
		\]
		It follows from \eqref{e7} that
		\[
		\|\varphi(\eta R|_{\mathcal N})\mathbf{x}\|
		=
		\|\varphi(\eta R)\mathbf{x}\|
		\geq\|\mathbf{x}\|,
		\qquad
		\mathbf{x}\in\mathcal N,\quad\eta\in\T.
		\]
		Thus $R|_{\mathcal N}$ satisfies all the
		hypotheses of Shimorin's criterion.
		Applying Theorem~\ref{t2}, we obtain
		\begin{equation}\label{e10}
			\mathcal N
			=
			[\mathcal N\ominus R\mathcal N]_R.
		\end{equation}
		
		\medskip
		\noindent
		\emph{Step 3. Identification of the wandering subspace.}
		
		Comparing \eqref{e8} and \eqref{e9},
		and using the orthogonality of the external
		direct sum, we obtain
		\begin{equation}\label{e11}
			\mathcal N\ominus R\mathcal N
			=
			(\mathcal M\ominus T^k\mathcal M)
			\oplus\{0\}^{\oplus(k-1)}.
		\end{equation}
		Let $g\in\mathcal M\ominus T^k\mathcal M$.
		For $n=qk+r$, where $0\leq r<k$, the vector
		\[
		R^n(g,0,\ldots,0)
		\]
		has $T^ng$ in the $r$th coordinate
		and zero in all other coordinates.
		
		Let
		$\pi_0:\mathcal H^{\oplus k}\to\mathcal H$
		be the coordinate projection onto the
		first component. Then
		\[
		\pi_0R^n(g,0,\ldots,0)
		=
		\begin{cases}
			(T^k)^qg, & n=qk,\\
			0, & k\nmid n.
		\end{cases}
		\]
		
		Now fix $f\in\mathcal M$.
		By \eqref{e10} and \eqref{e11},
		the vector $(f,0,\ldots,0)\in\mathcal N$
		belongs to the closed linear span of
		\[
		\left\{
		R^n(g,0,\ldots,0):
		n\geq0,\ 
		g\in\mathcal M\ominus T^k\mathcal M
		\right\}.
		\]
		Since $\pi_0$ is bounded, we obtain
		$
		f\in
		[\mathcal M\ominus T^k\mathcal M]_{T^k}.
		$
		Therefore
		\[
		\mathcal M
		\subset
		[\mathcal M\ominus T^k\mathcal M]_{T^k}.
		\]
		The reverse inclusion is immediate. Hence
		\[
		\mathcal M
		=
		[\mathcal M\ominus T^k\mathcal M]_{T^k}.
		\]
	\end{proof}

	\begin{cor}\label{c1}
		For $-1<\alpha\leq1$ and every integer $k\geq1$,
		the Beurling-type theorem holds for $ \mathcal{B}_{\alpha}^k$ on $A^2_\alpha(\D)$.
	\end{cor}

	\begin{proof}
		For $-1<\alpha\leq1$, the weight
		\[
		w_\alpha(z)
		=
		(\alpha+1)(1-|z|^2)^\alpha
		\]
		satisfies
		\[
		\Delta\log
		\frac{w_\alpha(z)}{1-|z|^2}
		=
		\frac{4(1-\alpha)}{(1-|z|^2)^2}
		\geq0,
		\]
		where
		$\Delta=\partial_x^2+\partial_y^2$.
	 By Shimorin's Proposition~3.3 and the subsequent
	 remark \cite{Shimorin2003},
	 the polynomial $\varphi(z)=1-(1-z)^3$
	 satisfies all the conditions of Theorem~\ref{t2}.
	 Since $\mathcal{B}_{\alpha}$ is an analytic contraction,
	 Proposition~\ref{p1} applies.
	\end{proof}

	\section{The negative range}\label{s3}
	
	In this section, we establish that for every $\alpha>1$,
	there is a nonzero closed $ \mathcal{B}_{\alpha}^k$-invariant subspace
	 that fails the wandering subspace property.
	
	For a finite set $A\subset\D$, write
	\[
	I_A
	=
	\{f\in A^2_\alpha:
	f(a)=0\ \text{for every }a\in A\},
	\]
	and let $K_A$ denote its reproducing kernel.

	Recall the following consequence of the proof of  \cite[Theorem~1.1]{LinXuYu2026}.
	\begin{thm}	\label{t3}
		For every $\alpha>1$ there exist a finite set
		$A\subset\D\setminus\mathbb R$,
		invariant under complex conjugation,
		and a $b\in(0,1)$ such that
		\begin{equation}\label{e12}
			K_A(b,0)<0.
		\end{equation}
	\end{thm}

	We first show that negative values of reproducing
	kernels can lead to the failure of the wandering
	subspace property under suitable equivalent Hilbert norms.
	
		Define the canonical conjugation by
	\[
	(Jf)(z)
	=
	\overline{f(\overline z)}.
	\]
	For $f\in\Hol(\D)$, let $\ord_0 f$ denote
	the order of vanishing of $f$ at the origin.
	\begin{prop}	\label{l1}
		Let $\HH$ be the space of the same analytic functions
		as $A^2_\alpha$ with an equivalent Hilbert norm
		for which $J$ is isometric.
		Suppose $E\subset\D\setminus\mathbb R$
		is a finite conjugation-invariant set
		and the reproducing kernel of
		\[
		I_E^{\HH}
		=
		\{f\in\HH:f|_E=0\}
		\]
		satisfies
		\begin{equation}\label{e13}
			K_E^{\HH}(x,y)<0
			\qquad
			\text{for some }0<x<y<1.
		\end{equation}
		Then for some $m\geq0$ the closed
		$M_z$-invariant subspace
		\begin{equation}\label{e14}
			\mathcal N_m
			=
			\{f\in I_E^{\HH}:\ord_0 f\geq m\}
		\end{equation}
	has a one dimensional wandering subspace
	but fails the wandering subspace property.
	\end{prop}

	\begin{proof}
		Since point and derivative evaluations are continuous
		under the equivalent norm, each $\mathcal N_m$ is closed.
		
		We first show that
		\begin{equation}\label{e15}
			z\mathcal N_m=\mathcal N_{m+1}.
		\end{equation}
		The conclusion of
		$z\mathcal N_m\subseteq\mathcal N_{m+1}$ is obvious.
			Conversely, if $f\in\mathcal N_{m+1}$,
		then $f(0)=0$,    then there exists
		a constant $C>0$ such that
		\[
		\|f/z\|_{\HH}
		\leq C
		\|f\|_{\HH}.
		\]
		Since $0\notin E$, the function $f/z$
		vanishes on $E$ and has a zero of order
		at least $m$ at the origin.
		Hence $f/z\in\mathcal N_m$, proving
		\eqref{e15}.
		
		In particular, each $\mathcal N_m$ is
		$M_z$-invariant, and
		\[
		\mathcal N_m\ominus z\mathcal N_m
		=
		\mathcal N_m\ominus\mathcal N_{m+1}.
		\]
		
		Now, consider the continuous linear functional
		\[
		f\longmapsto f^{(m)}(0)
		\]
		on $\mathcal N_m$.
		Its kernel is precisely $\mathcal N_{m+1}$.
	   Since the functional is nonzero,
	it follows from the isomorphism theorem that
	\[
	\dim(\mathcal N_m/\mathcal N_{m+1})=1.
	\]
		Since $\mathcal N_{m+1}$ is closed,
		there exists a unit vector
		$g_m\in\mathcal N_m\ominus\mathcal N_{m+1}$
		such that
		\[
		\mathcal N_m\ominus z\mathcal N_m
		=
		\mathbb C g_m.
		\]
		Multiplying $g_m$ by a unimodular constant,
		we may assume that
		\[
		g_m^{(m)}(0)>0.
		\]

	Since $E$ is invariant under complex conjugation,
	we have
	\[
	J(\mathcal N_m)=\mathcal N_m,
	\qquad m\geq0.
	\]
	As $J$ is an isometric conjugation, it preserves
	the one dimensional subspace
	$\mathcal N_m\ominus\mathcal N_{m+1}$.
	Consequently, there exists $\lambda\in\mathbb T$
	such that
	\[
	Jg_m=\lambda g_m.
	\]
	Taking the $m$th derivative at the origin gives
	\[
	\overline{g_m^{(m)}(0)}
	=
	\lambda g_m^{(m)}(0).
	\]
	Since $g_m^{(m)}(0)>0$, we obtain $\lambda=1$.
	Thus
	\begin{equation}\label{e16}
		Jg_m=g_m,
		\qquad
		g_m(t)\in\mathbb R
		\quad (-1<t<1).
	\end{equation}
	
We next show that $\{g_m\}_{m\geq0}$ is an
orthonormal basis for $\mathcal N_0$.
By \eqref{e15}, we have
$\mathcal N_m=M_z^m\mathcal N_0$ for every $m\geq0$.
Since $M_z$ is analytic,
\[
\bigcap_{m\geq0}\mathcal N_m=\{0\}.
\]
Consequently, the orthogonal decompositions
\[
\mathcal N_m
=\mathbb Cg_m\oplus\mathcal N_{m+1},
\qquad m\geq0,
\]
yield
\[
\mathcal N_0
=\bigoplus_{m=0}^{\infty}\mathbb Cg_m.
\]
Since $\|g_m\|_{\HH}=1$, the family
$\{g_m\}_{m\geq0}$ is an orthonormal basis
for $\mathcal N_0$.
	
	It follows from the orthonormal basis expansion
	of the reproducing kernel that
	\begin{equation}\label{e161}
		K_E^{\HH}(z,w)
		=
		\sum_{m=0}^{\infty}
		g_m(z)\overline{g_m(w)},
		\qquad z,w\in\D.
	\end{equation}
	The series converges absolutely for each
	$(z,w)\in\D\times\D$.
	
	In particular, for $x,y\in(-1,1)$,
	\eqref{e16} and \eqref{e161} give that
	\[
	K_E^{\HH}(x,y)
	=
	\sum_{m=0}^{\infty}g_m(x)g_m(y)
	\]
 is real. Fix $x,y\in(-1,1)$ satisfies \eqref{e13}, then there exists
$m\geq0$ such that
\[
g_m(x)g_m(y)<0.
\]
By continuity, there exists $c\in(x,y)$
such that $g_m(c)=0$.

Finally, note that
\[
[\mathcal N_m\ominus z\mathcal N_m]_{M_z}
=
\overline{\operatorname{span}}
\{z^n g_m:n\geq0\}.
\]
Therefore, every function
in this closed linear span vanishes at $c$.
		
		On the other hand, the polynomial
		\[
		h_m(z)=z^m\prod_{a\in E}(z-a)
		\]
		belongs to $\mathcal N_m$ and satisfies
		\[
		h_m(c)=c^m\prod_{a\in E}(c-a)\neq0,
		\]
		since $c\in(0,1)$ and $E\cap\mathbb R
		=\varnothing$.
		
		Consequently,
		\[
		[\mathcal N_m\ominus z\mathcal N_m]_{M_z}
		\subsetneq\mathcal N_m.
		\]
	\end{proof}

The preceding proposition shows that negative values
of reproducing kernels can lead to the failure of
the wandering subspace property. We next show that
such kernel negativity persists under equivalent
inner products induced by compact perturbations
of the identity that commute with conjugation.

\begin{lem}\label{l2}
	Let $Q$ be a positive invertible operator on
	$A^2_\alpha$ satisfying
	\begin{equation}\label{e17}
		Q-I \text{ is compact},
		\qquad QJ=JQ.
	\end{equation}
	Equip $A^2_\alpha$ with the equivalent inner product
	\[
	\langle f,g\rangle_Q
	=\langle Qf,g\rangle_\alpha.
	\]
	Then there exist a finite set
	$E\subset\D\setminus\mathbb R$, invariant under
	complex conjugation, and $0<x<y<1$
	such that the reproducing kernel $K_E^Q$
	of $I_E$
	satisfies
	\begin{equation}\label{e18}
		K_E^Q(x,y)<0.
	\end{equation}
\end{lem}

	\begin{proof}
	For $a\in\D$, define
	\begin{equation}\label{e19}
		\varphi_a(z)
		=
		\frac{a-z}{1-\overline a z},
		\quad
		j_a(z)
		=
		\frac{(1-|a|^2)^{(\alpha+2)/2}}
		{(1-\overline a z)^{\alpha+2}},
		\quad
		U_af=j_a(f\circ\varphi_a).
	\end{equation}
	The standard identities for disk automorphisms
	(see \cite[Chapter~4]{Zhu}) yield
	\[
	(1-|\varphi_a(z)|^2)^\alpha
	|\varphi_a'(z)|^2
	=
	|j_a(z)|^2(1-|z|^2)^\alpha.
	\]
	In what follows, we take $a=t\in(0,1)$. A change of variables shows that
	$U_t$ is an isometry on $A^2_\alpha$.
	Since $\varphi_t\circ\varphi_t=\mathrm{id}$
	and $j_t(j_t\circ\varphi_t)=1$,
	we have $U_t^2=I$.
	Thus $U_t$ is a
	self-adjoint unitary. Since $t$ is real,
	\[
	\varphi_t(\overline z)
	=\overline{\varphi_t(z)},
	\qquad
	j_t(\overline z)
	=\overline{j_t(z)},
	\]
	and hence $U_tJ=JU_t$.
	Moreover,
	\begin{equation}\label{e20}
		j_t(s)>0,\qquad -1<s<1.
	\end{equation}

We next show that $U_tf$ converges weakly to zero as $t\to1$
for every $f\in A^2_\alpha$.
For  polynomial $f$, we have $U_tf\to0$
uniformly on compact subsets of $\D$, since
$j_t\to0$ uniformly on compact subsets of $\D$
and $f\circ\varphi_t$ is uniformly bounded on $\D$.
By Cauchy's integral formula, every Taylor
coefficient of $U_tf$ tends to zero.
Since the monomials form an orthogonal basis
of $A^2_\alpha$ and the operators $U_t$ are unitary,
it follows that $U_tf\to0$ weakly
for every polynomial $f$.
By the density of polynomials, the same conclusion holds for every $f\in A^2_\alpha$.

	Since $Q-I$ is compact and $U_tf\to0$ weakly
	for every $f\in A^2_\alpha$, we obtain
	\begin{equation}\label{e21}
		Q_t:=U_t^*QU_t
		\xrightarrow{\mathrm{SOT}} I
		\qquad \text{as } t\to1.
	\end{equation}
Since  $Q$ is a positive invertible operator, then there exists constants $0<a\leq c<\infty$ such that
	\begin{equation}\label{e22}
		aI\leq Q_t\leq cI,
		\qquad 0<t<1.
	\end{equation}
Moreover, $Q_t$ commute with $J$.
	
	Choose a finite set $A\subset\D\setminus\mathbb R$
	and $b\in(0,1)$ as in Theorem~\ref{t3}, such that
	\[
	K_A(b,0)<0.
	\]
	Let $P_A$ denote the orthogonal projection onto
	$I_A$ in $A^2_\alpha$, and define
	\[
	C_t=P_AQ_t|_{I_A}.
	\]
	By \eqref{e21} and \eqref{e22},
	\[
	C_t\xrightarrow{\mathrm{SOT}} I_{I_A}
	\qquad \text{as } t\to1,
	\]
	with
	\[
	aI_{I_A}\leq C_t\leq cI_{I_A}.
	\]
Since $\|C_t^{-1}\|\leq a^{-1}$, we have
\[
\|(C_t^{-1}-I)f\|_\alpha
\leq a^{-1}\|(I-C_t)f\|_\alpha
\longrightarrow0
\]
for every $f\in I_A$. Hence
\[
C_t^{-1}\xrightarrow{\mathrm{SOT}}I_{I_A}
\qquad \text{as }t\to1.
\]
	
		Let $K_{A,t}$ denote the reproducing kernel of
		$I_A$ with respect to the $Q_t$-inner product.
		For $f\in I_A$, we have
		\[
		\langle f,C_t^{-1}K_A(\cdot,w)\rangle_{Q_t}
		=
		\langle f,K_A(\cdot,w)\rangle_\alpha
		=f(w).
		\]
		Thus
		\begin{equation}\label{e23}
			K_{A,t}(\cdot,w)
			=
			C_t^{-1}K_A(\cdot,w).
		\end{equation}
		The strong convergence $C_t^{-1}\xrightarrow{\mathrm{SOT}}I_{I_A}$ implies
		\begin{equation}\label{e24}
			K_{A,t}(b,0)
			\longrightarrow K_A(b,0)<0.
		\end{equation}
		Moreover, since $JI_A=I_A$ and $Q_tJ=JQ_t$,
		the reproducing property implies that
		$JK_{A,t}(\cdot,w)=K_{A,t}(\cdot,w)$
		for every real $w\in\D$.
		In particular, $K_{A,t}(b,0)$ is real.
		Hence, for $t>b$ sufficiently close to $1$,
		\[
		K_{A,t}(b,0)<0.
		\]
		
		Since $Q_t=U_t^*QU_t$, the operator $U_t$
		is unitary from the $Q_t$-space onto the
		$Q$-space. Moreover, it maps $I_A$
		onto $I_E$, where $E=\varphi_t(A)$.
		The set $E$ is finite, invariant under
		complex conjugation, and disjoint from $\mathbb R$.
		The corresponding reproducing kernels satisfy
		\begin{equation}\label{e25}
			K_E^Q(\xi,\eta)
			=
			j_t(\xi)\overline{j_t(\eta)}
			K_{A,t}\bigl(
			\varphi_t(\xi),\varphi_t(\eta)
			\bigr).
		\end{equation}
		
		Finally, set
		\[
		x=\varphi_t(b)=\frac{t-b}{1-tb},
		\qquad
		y=\varphi_t(0)=t.
		\]
		Then $0<x<y<1$, and by \eqref{e20} and
		\eqref{e25},
		\[
		K_E^Q(x,y)
		=
		j_t(x)j_t(y)K_{A,t}(b,0)<0.
		\]
	\end{proof}

We now combine Proposition~\ref{l1} and Lemma~\ref{l2} to construct invariant subspaces
that fails the wandering subspace property.

For each integer $k\geq1$, let
\[
\E_k
=
\overline{\operatorname{span}}
\{z^{kn}:n\geq0\}
\subset A^2_\alpha.
\]
\begin{thm}\label{t4}
	For every $\alpha>1$ and every integer $k\geq1$
	there is a nonzero closed
	$\mathcal{B}_{\alpha}^k$-invariant subspace
	$\mathcal M\subset\E_k$ such that
	\[
	\dim(
	\mathcal M\ominus \mathcal{B}_{\alpha}^k\mathcal M
	)=1,
	\]
	and
	\[
	[\mathcal M\ominus \mathcal{B}_{\alpha}^k\mathcal M]_{\mathcal{B}_{\alpha}^k}
	\subsetneq\mathcal M.
	\]
\end{thm}

\begin{proof}
	For $F(z)=\sum_{n\geq0}a_nz^n$, define
	\begin{equation}\label{e26}
		\HH_{\alpha,k}
		=
		\{F\in\Hol(\D):
		\|F\|_{\alpha,k}<\infty\},
	\end{equation}
	with $\|F\|_{\alpha,k}^2
	=
	k^{\alpha+1}
	\sum_{n\geq0}
	|a_n|^2\|z^{kn}\|_\alpha^2$.
	
	Recall that
	\[
	\|z^n\|_\alpha^2
	=
	\frac{\Gamma(\alpha+2)\Gamma(n+1)}
	{\Gamma(n+\alpha+2)}
	\sim
	\Gamma(\alpha+2)n^{-(\alpha+1)}
	\qquad (n\to\infty).
	\]
	Hence
	\[
	q_n
	:=
	k^{\alpha+1}
	\frac{\|z^{kn}\|_\alpha^2}
	{\|z^n\|_\alpha^2}
	\longrightarrow1.
	\]
	Define $Q$ on $A^2_\alpha$ by
	\[
	Qz^n=q_nz^n,
	\qquad n\geq0.
	\]
	Since $q_n>0$ and $q_n\to1$,
	$Q$ is positive and invertible, and
	$Q-I$ is compact.
	Moreover, $QJ=JQ$, and
	\[
	\langle F,G\rangle_{\alpha,k}
	=
	\langle QF,G\rangle_\alpha.
	\]
	Thus $\HH_{\alpha,k}$ and $A^2_\alpha$
	consist of the same functions with equivalent Hilbert space norms.
	
Define $V_k:\HH_{\alpha,k}\to\E_k$ by
\begin{equation}\label{e30}
	(V_kF)(z)
	=
	k^{(\alpha+1)/2}F(z^k).
\end{equation}
By the definition of $\|\cdot\|_{\alpha,k}$,
\[
\|V_kF\|_\alpha
=
\|F\|_{\alpha,k}.
\]
Thus $V_k$ is an isometry. Since its range is
closed and contains every monomial $z^{kn}$,
$n\geq0$, it follows that $V_k$ is unitary
from $\HH_{\alpha,k}$ onto $\E_k$.

Moreover, for every $F\in\HH_{\alpha,k}$,
we have $V_k(zF)=z^kV_kF$. Hence
\begin{equation}\label{e31}
	V_kM_z=(\mathcal{B}_{\alpha}^k|_{\E_k})V_k.
\end{equation}
Consequently, $M_z$ on $\HH_{\alpha,k}$
is unitarily equivalent to $\mathcal{B}_{\alpha}^k|_{\E_k}$.
	
	Applying Lemma~\ref{l2}, there exist
	a finite conjugation-invariant set
	$E\subset\D\setminus\mathbb R$
	and $0<x<y<1$ such that
	\[
	K_E^{\HH_{\alpha,k}}(x,y)<0.
	\]
By Proposition~\ref{l1}, there exists
$m\geq0$ such that
\[
\mathcal N_m
=
\{F\in\HH_{\alpha,k}:
F|_E=0,\ \ord_0F\geq m\}
\]
is a nonzero closed $M_z$-invariant subspace
satisfying
\[
\dim(\mathcal N_m\ominus z\mathcal N_m)=1,
\qquad
[\mathcal N_m\ominus z\mathcal N_m]_{M_z}
\subsetneq\mathcal N_m.
\]

Set $\mathcal M=V_k\mathcal N_m$.
Since $V_k$ is unitary and satisfies \eqref{e31},
$\mathcal M$ is a nonzero closed
$\mathcal B_\alpha^k$-invariant subspace of $\E_k$.
Moreover,
\[
\mathcal M\ominus\mathcal B_\alpha^k\mathcal M
=
V_k(\mathcal N_m\ominus z\mathcal N_m),
\]
and
\[
[\mathcal M\ominus\mathcal B_\alpha^k\mathcal M]
_{\mathcal B_\alpha^k}
=
V_k[\mathcal N_m\ominus z\mathcal N_m]_{M_z}
\subsetneq\mathcal M.
\]
In particular,
\[
\dim(\mathcal M\ominus\mathcal B_\alpha^k\mathcal M)=1.
\]
\end{proof}

	\begin{remark}
		The compact perturbation argument in the proof of
		Theorem~\ref{t4} does not depend on the particular
		choice of the sequence $(q_n)$.
		In fact, let $(\lambda_n)_{n\geq0}$ be any sequence
		of positive numbers satisfying
		\[
		\frac{\lambda_n}{\|z^n\|_\alpha^2}
		\longrightarrow 1
		\qquad (n\to\infty).
		\]
		Consider the norm
		\[
		\left\|\sum_{n=0}^\infty a_nz^n\right\|_\lambda^2
		=
		\sum_{n=0}^\infty |a_n|^2\lambda_n.
		\]
		The corresponding operator
		\[
		Q_\lambda z^n
		=
		\frac{\lambda_n}{\|z^n\|_\alpha^2}z^n
		\]
		is positive and invertible, satisfies
		$Q_\lambda-I$ is compact, and commutes with $J$.
		Thus the preceding compact perturbation argument
		applies to any such weighted norm.
		The particular choice of $(q_n)$ is needed only
		for the unitary identification with
		$\mathcal{B}_{\alpha}^k|_{\mathcal E_k}$.
	\end{remark}

\begin{proof}[Proof of Theorem~\ref{t1}]
	By Corollary~\ref{c1},  the Beurling-type theorem holds for
	$-1<\alpha\leq1$ and every $k\geq1$.
	Conversely, for each $\alpha>1$ and
	$k\geq1$, Theorem~\ref{t4} provides
	an $\mathcal{B}_{\alpha}^k$-invariant subspace that fails wandering subspace property.
	This completes the proof.
\end{proof}

\section{Failure of the wandering subspace property for finite Blaschke multipliers}\label{s4}

The Beurling-type theorem  in Theorem~\ref{t1}
does not extend to finite Blaschke multipliers.
More precisely, for every $\alpha>-1$, we construct a finite
Blaschke multiplier admitting a cyclic invariant subspace
with a one dimensional wandering subspace but without  wandering subspace property. 

For $a\in\D$, let
\[
b_a(z)=\frac{z-a}{1-\overline a z}=-\varphi_{a}(z),
\qquad
B_{k,a}(z)=b_a(z^k)
=\frac{z^k-a}{1-\overline a z^k}.
\]
We retain the notation $\E_k$, $\HH_{\alpha,k}$,
and $V_k:\HH_{\alpha,k}\to\E_k$ from
Section~\ref{s3}. 
By the definition of
$\HH_{\alpha,k}$,
its reproducing kernel is given by, see \cite[Theorem 4.19]{Zhu},
\begin{equation}\label{e35}
	K_{\alpha,k}(z,w)
	=k^{-(\alpha+1)}
	\sum_{n=0}^\infty
	\frac{(z\overline w)^n}{\|z^{kn}\|_\alpha^2}.
\end{equation}
Since $\HH_{\alpha,k}$ and $A^2_\alpha$
have equivalent norms, every function in
$H^\infty(\D)$ is a bounded multiplier of
$\HH_{\alpha,k}$.

We first determine the wandering subspace of
$M_{b_a}$ on $\HH_{\alpha,k}$ and establish
the cyclicity of the constant function $1$.
\begin{lem}\label{l3}
	Let $\alpha>-1$, $k\geq1$, and $a\in\D$.
	Then
	\begin{equation}\label{e36}
		b_a\HH_{\alpha,k}
		=\{F\in\HH_{\alpha,k}:F(a)=0\}.
	\end{equation}
	Moreover,
	\[
	\HH_{\alpha,k}\ominus b_a\HH_{\alpha,k}
	=\mathbb C K_{\alpha,k}(\cdot,a),
	\]
	and
	\begin{equation}\label{e37}
		\HH_{\alpha,k}=[1]_{M_{b_a}}.
	\end{equation}
\end{lem}

\begin{proof}
	The inclusion from left to right in \eqref{e36}
	is immediate. Conversely, suppose that $F(a)=0$.
	By the unitary operator $U_a$ introduced
	in Section~\ref{s3}, we have
	\[
	U_a(F/b_a)=-(U_aF)/z,
	\qquad
	(U_aF)(0)=0.
	\]
	 It follows that
	\[
	\left\|\frac{F}{b_a}\right\|_\alpha
	\leq \sqrt{\alpha+2}\,\|F\|_\alpha.
	\]
	By norm equivalence,
	$F/b_a\in\HH_{\alpha,k}$.
	This proves \eqref{e36}. The assertion concerning
	the orthogonal complement follows from the
	reproducing property of $K_{\alpha,k}$.
	
	To prove \eqref{e37}, observe that
	\[
	z=\frac{b_a(z)+a}{1+\overline a b_a(z)}.
	\]
	Since $|a|<1$, for every $n\geq1$,
	the Taylor series of
	\[
	\left(\frac{w+a}{1+\overline a w}\right)^n
	\]
	converges uniformly on the closed unit disk.
	Consequently, each monomial $z^n$ is a uniform
	limit on $\D$ of polynomials in $b_a$.
	By norm equivalence, this convergence also
	holds in $\HH_{\alpha,k}$.
	Since polynomials are dense in $\HH_{\alpha,k}$,
	we obtain \eqref{e37}.
\end{proof}

We now use the reproducing kernel formula
\eqref{e35} to construct zeros that lead to
the failure of the wandering subspace property
for finite Blaschke multipliers.

\begin{thm}\label{t5}
	For every $\alpha>-1$, there exists an integer
	$N_\alpha\geq1$ such that, for every $k\geq N_\alpha$,
	one can choose $a\in(0,1)$ for which the finite
	Blaschke product
	\[
	B_{k,a}(z)=\frac{z^k-a}{1-az^k}
	\]
	admits a closed invariant subspace $\E_k$ satisfying
	\[
	\E_k=[1]_{M_{B_{k,a}}},
	\qquad
	\dim(\E_k\ominus B_{k,a}\E_k)=1,
	\]
	but
	\begin{equation}\label{e38}
		[\E_k\ominus B_{k,a}\E_k]_{M_{B_{k,a}}}
		\subsetneq\E_k.
	\end{equation}
\end{thm}

\begin{proof}
	Fix $\alpha>-1$.
	First note that
	$\|z^k\|_\alpha^2\to0$ as $k\to\infty$.
	Moreover, for $n\geq1$,
	\begin{equation}\label{e39}
		\begin{aligned}
			\frac{\|z^k\|_\alpha^2}
			{\|z^{kn}\|_\alpha^2}
			&=\prod_{j=k+1}^{kn}
			\left(1+\frac{\alpha+1}{j}\right)\\
			&\leq n^{\alpha+1}.
		\end{aligned}
	\end{equation}
	
	For all sufficiently large $k$, we have
	$\|z^k\|_\alpha^2\leq1/8$.
	By \eqref{e39},
	\[
	\begin{aligned}
		&\left|
		\sum_{n=2}^\infty
		\frac{(-2\|z^k\|_\alpha^2)^n}
		{\|z^{kn}\|_\alpha^2}
		\right|\\
		&\qquad\leq
		4\|z^k\|_\alpha^2
		\sum_{n=2}^\infty
		n^{\alpha+1}4^{-(n-2)}
		\longrightarrow0
		\qquad(k\to\infty).
	\end{aligned}
	\]
	Consequently, there exists an integer
	$N_\alpha\geq1$ such that, for every
	$k\geq N_\alpha$,
	\[
	\sum_{n=0}^\infty
	\frac{(-2\|z^k\|_\alpha^2)^n}
	{\|z^{kn}\|_\alpha^2}<0.
	\]
	On the other hand,
	\[
	\sum_{n=0}^\infty
	\frac{0^n}{\|z^{kn}\|_\alpha^2}=1.
	\]
	By continuity and the intermediate value theorem,
	there exists
	$t_0\in(-2\|z^k\|_\alpha^2,0)$
	such that
	\[
	\sum_{n=0}^\infty
	\frac{t_0^n}{\|z^{kn}\|_\alpha^2}=0.
	\]
	
	Choose $a\in(|t_0|,1)$ and set
	$\xi=t_0/a\in(-1,0)$. By \eqref{e35},
	\begin{equation}\label{e40}
		K_{\alpha,k}(\xi,a)
		=k^{-(\alpha+1)}
		\sum_{n=0}^\infty
		\frac{t_0^n}{\|z^{kn}\|_\alpha^2}
		=0.
	\end{equation}
	
	Since $V_k$ is unitary from
	$\HH_{\alpha,k}$ onto $\E_k$ and
	\[
	V_kM_{b_a}
	=(M_{B_{k,a}}|_{\E_k})V_k,
	\]
	Lemma~\ref{l3} gives
	\[
	\E_k
	=V_k[1]_{M_{b_a}}
	=[V_k1]_{M_{B_{k,a}}}.
	\]
	Since $V_k1=k^{(\alpha+1)/2}1$, it follows that
	\[
	\E_k=[1]_{M_{B_{k,a}}}.
	\]
	Furthermore,
	\[
	\E_k\ominus B_{k,a}\E_k
	=V_k(\HH_{\alpha,k}\ominus b_a\HH_{\alpha,k})
	=\mathbb C V_kK_{\alpha,k}(\cdot,a).
	\]
	
	Choose $z_0\in\D$ such that $z_0^k=\xi$.
	By \eqref{e30} and \eqref{e40},
	\[
	(V_kK_{\alpha,k}(\cdot,a))(z_0)=0.
	\]
	Hence every polynomial multiple of
	$V_kK_{\alpha,k}(\cdot,a)$ by $B_{k,a}$
	vanishes at $z_0$.
	Since evaluation at $z_0$ is bounded on
	$A^2_\alpha$, it follows that
	\[
	f(z_0)=0
	\qquad
	\text{for all } f\in
	[\E_k\ominus B_{k,a}\E_k]_{M_{B_{k,a}}}.
	\]
	Since $1\in\E_k$, the inclusion in
	\eqref{e38} is proper.
\end{proof}

We end  this section with an explicit example exhibiting such an integer $k$.
	\begin{example}\label{ex1}
		Consider the finite Blaschke product
		\[
		B(z)=\frac{z^3-\frac34}{1-\frac34z^3}
		\]
		on the classical Bergman space $A^2(\D)$.
		Let
		\[
		\E_3=\overline{\operatorname{span}}
		\{z^{3n}:n\geq0\}.
		\]
		In the sequel, we show that $\E_3$ is a cyclic $M_B$-invariant
		subspace satisfying
		\[
		\E_3=[1]_{M_B},
		\qquad
		\dim(\E_3\ominus B\E_3)=1,
		\]
		but
		\[
		[\E_3\ominus B\E_3]_{M_B}
		\subsetneq\E_3.
		\]
		
		Set
		$a=3/4$ and $r=a^{1/3}$.
		Since $B=b_a\circ z^3$, Lemma~\ref{l3} and equation
		\[
		V_3M_{b_a}=(M_B|_{\E_3})V_3
		\]
		give
		\[
		\E_3=[1]_{M_B},
		\qquad
		B\E_3=\{f\in\E_3:f(r)=0\}.
		\]
		
		Since $\|z^{3n}\|_0^2=(3n+1)^{-1}$,
		the reproducing kernel of $\E_3$ is
		\[
		\begin{aligned}
			K_{\E_3}(z,w)
			&=\sum_{n=0}^{\infty}
			(3n+1)(z\overline w)^{3n}\\
			&=\frac{1+2(z\overline w)^3}
			{\bigl(1-(z\overline w)^3\bigr)^2}.
		\end{aligned}
		\]
		Consequently,
		\[
		\E_3\ominus B\E_3
		=\mathbb C K_{\E_3}(\cdot,r)
		=\mathbb C g,
		\]
		where
		\[
		g(z)=
		\frac{1+\frac32z^3}
		{\left(1-\frac34z^3\right)^2}.
		\]
		
	Clearly, $g$ has three zeros in $\D$,  and every function in
		\[
		[\E_3\ominus B\E_3]_{M_B}
		=
		\overline{\operatorname{span}}
		\{B^ng:n\geq0\}
		\]
		vanishes at $\sqrt[3]{-\frac{2}{3}}$.
		On the other hand, $1\in\E_3$.
		Therefore,
		\[
		[\E_3\ominus B\E_3]_{M_B}
		\subsetneq\E_3.
		\]
	\end{example}

\medskip\noindent\emph{Remark.}
Theorem~\ref{t5} contrasts with Theorem~\ref{t1}
even on the same invariant subspace. Indeed, for
$-1<\alpha\leq1$, the subspace $\E_k$ is invariant
under both $\mathcal{B}_{\alpha}^k$ and $M_{B_{k,a}}=b_a(\mathcal{B}_{\alpha}^k)$, and
\[
\E_k=[\E_k\ominus\mathcal{B}_{\alpha}^k\E_k]_{\mathcal{B}_{\alpha}^k},
\]
whereas
\[
[\E_k\ominus B_{k,a}\E_k]_{M_{B_{k,a}}}
\subsetneq\E_k.
\]
Thus the wandering subspace property is not
necessarily preserved under Möbius transformations
of the operator.

	\noindent\textbf{AI disclosure.}
	The authors used ChatGPT-6 (OpenAI) to assist with language refinement, exposition, \LaTeX{} preparation,
	and the exploration and checking of selected local mathematical arguments.
   More specifically, the proof of the cyclicity assertion
   $\mathcal H_{\alpha,k}=[1]_{M_{b_a}}$
   in Lemma~\ref{l3} was provided by GPT-6.
	The research questions and main mathematical ideas in
	this article originated with the authors.
	All mathematical statements, proofs, and conclusions were independently
	checked by the authors, who take full responsibility for the content
	of the paper.

\end{document}